\documentclass{birkau}
\usepackage{enumerate}
\usepackage{soul}
\usepackage{tabularx}
\usepackage{centernot}
\usepackage{mathtools}
\usepackage{stmaryrd}
\usepackage{etoolbox}
\usepackage{tikz}
\usepackage{tablefootnote}
\usetikzlibrary{matrix}
\usepackage{graphics,graphicx}
\usepackage[all]{xy}
\usepackage{tikz-cd}
\usepackage{url}
\usepackage{marginnote}
\definecolor{mygray}{gray}{0.85}
\usepackage[linecolor=black,backgroundcolor=mygray,colorinlistoftodos,prependcaption,textsize=small]{todonotes}
\usepackage{xargs}

\usepackage{amsmath,amssymb,url}
\usepackage{enumitem} 

\numberwithin{equation}{section}

\theoremstyle{plain}
\newtheorem{theorem}{Theorem}[section]
\newtheorem{lemma}[theorem]{Lemma}
\newtheorem{question}[theorem]{Question}
\newtheorem{proposition}[theorem]{Proposition}
\newtheorem{construction}[theorem]{Construction}
\newtheorem{corollary}[theorem]{Corollary}

\newtheorem{context}[theorem]{Context}
\newtheorem{fact}[theorem]{Fact}

\theoremstyle{definition}
\newtheorem{definition}[theorem]{Definition}
\newtheorem{notation}[theorem]{Notation}
\newtheorem{remark}[theorem]{Remark}

\renewcommand{\leq}{\leqslant}
\renewcommand{\geq}{\geqslant}

\newcommand\myrestriction{\mathord\restriction}

\def\ZZ{{\mathbb Z}}

\newcommand{\NN}{\mathbb{N}}

\newcommand{\Cscr}{\mathcal C}

\newcommand{\Gscr}{\mathcal G}

\newcommand{\Tscr}{\mathcal T}
\newcommand{\Uscr}{\mathcal U}

\def\abar{\mbox{\boldmath $a$}}
\def\bbar{{\bf b}}

\def\hbar{{\bf h}}

\def\xbar{{\bf x}}
\def\ybar{{\bf y}}

\def\bU{{\bf U}}
\def\bL{{\bf L}}
\def\bK{{\bf K}}

\def\dcl{\rm dcl}

\def\acl{\rm acl}
\def\th{\rm Th}

\makeatletter
\def\subsection{\@startsection{subsection}{3}%
  \z@{.5\linespacing\@plus.7\linespacing}{.3\linespacing}%
  {\bfseries\centering}}
\makeatother

\makeatletter
\def\subsubsection{\@startsection{subsubsection}{3}%
  \z@{.5\linespacing\@plus.7\linespacing}{.3\linespacing}%
  {\centering}}
\makeatother

\date{\today}
\begin{document}

\title[Coordinatization and Quasigroups]{Strongly minimal Steiner Systems II: Coordinatization and
quasigroups}

\corrauthor[John T. Baldwin]{John T. Baldwin}
\address{Department of Mathematics, Statistics and Computer Science\\
University of Illinois at Chicago\\Chicago, Illinois 60607\\USA}
\urladdr{http://http://homepages.math.uic.edu/~jbaldwin/}
\email{jbaldwin@uic.edu}

\thanks{\thanks{Research partially supported by Simons
travel grant  G3535.}}



\begin{abstract}
Each strongly minimal  Steiner $k$-system  $(M,R)$ (where is $R$ is a
ternary collinearity relation) can be `coordinatized' in the sense of
(Ganter-Werner 1975) by a quasigroup if $k$ is a prime-power. We show this
coordinatization is never definable in $(M,R)$ and the strongly minimal
Steiner $k$-systems constructed in (Baldwin-Paolini 2020) never interpret a
quasigroup. Nevertheless, by refining the construction, if $k$ is a prime
power, in each $(2,k)$-variety of quasigroups (Definition 3.10) there is a
 strongly minimal quasigroup that  interprets a Steiner
 $k$-system.\end{abstract}

\maketitle

Data sharing not applicable to this article as no datasets were generated
or analysed during the current study.

\section{Introduction} \label{intro}
\begin{quote}
Steiner Triple Systems are in a 1-1-correspondence with the so-called
squags (Steiner quasigroups: groupoids satisfying the identities $xx=x,
xy = yx, x(xy)= y)$. With the help of this correspondence, many
combinatorial properties of Steiner Triple Systems can be described in an
algebraic language, and algebraic methods have successfully been applied.

Ganter and Werner \cite[opening paragraph]{GanWer}
\end{quote}

 A \emph{ linear space} is collection of points and lines that satisfy a
minimal condition to call a structure a geometry: two points determine a
line.  For us, a \emph{ Steiner $k$-system} is a linear space such that
every line (block) has cardinality $k>2$. A quasigroup is a structure with
a single binary operation whose multiplication table is a Latin square
(each row or column is a permutation of the universe).

The main contribution of \cite{GanWer} is to generalize the correspondence
they describe to Steiner $q$-systems when $q=p^n$ for some $n$ and show
these are the only Steiner $k$-systems in our sense that can be
coordinatized. (They also classify the coordinatization of Steiner systems
where a block is determined by $k$ elements with $k> 2$.)

We combine  methods developed to study $\aleph_1$-categorial first order
theories to construct new families of quasigroups with the methods of
 universal algebra \cite{GanWer} to  study interpretability between
(coordinatization of) Steiner systems and quasigroups. Section~\ref{bg}
lays out model theoretic background on the general Hrushovski method
(Section~\ref{gencon}), the special case that yields Steiner systems
(Section~\ref{genspace}), and the notion of interpretation
(Section~\ref{coord}). In Section~\ref{assqg} we distinguish interpretation
from coordinatization of Steiner systems (e.g. \cite{GanWer}) and show the
strongly mimimal Steiner systems are interpretable in incomplete theories
of quasigroups but not conversely. In Section~\ref{findqg}, we extend the
Hrushovski technology to construct strongly minimal quasigroups.  We
conclude with universal algebraic questions which arise from this
construction.

Steiner $k$-systems are generally considered in a $2$-sorted vocabulary
with sorts for points and lines; for reasons discussed in
Remark~\ref{1sort}, we use the bi-interpretable (\cite{BaldwinPao}) setting
of a single sorted structure $(M,R)$ with one ternary `collinearity'
relation for both linear spaces and Steiner systems.
 Such mathematicians as Steiner, Bose, Skolem, and Bruck have established
deep connections between the existence of a Steiner system with $v$ points
and blocks of size $k$ and divisibility relations among $k$ and $v$. This
interaction with number theory is reflected in a line of work from the
1950-1980's including \cite{Steinpnas,Gratdt, GanWer,Evans2var}. It
culminates with the proof  that Steiner $k$-systems are `coordinatized' by
varieties of quasigroups if and only if $k$ is a prime power, $q$. We
consider Steiner systems of every infinite cardinality with two contrasting
results. Building on Fact~\ref{disp2a}, we prove Theorem~\ref{disp2} as
Theorem~\ref{coordthm} of this paper. Below, we label as `facts' theorems
from our earlier papers that are used here.

\begin{fact}[\cite{BaldwinPao}]\label{disp2a} For each $k\geq 3$ there are
uncountably many $\mu$  and an associated theory $T_\mu$ such that $T_\mu$
    is the theory of a strongly minimal Steiner $k$-system $(M,R)$. $\mu$
    is a function into the natural numbers counting the realizations of
    good pairs (Definition~\ref{Kmu}).
\end{fact}

The argument for the Theorem~\ref{disp2} (given in Section~\ref{assqg} is
heavily based on Theorem ~\ref{coordthm} \cite{GanWer}, where
`coordinatized' is a particular form of `interpreted'. Mikado varieties are
described in Definition~\ref{rkdef}. Item 2) is immediate since the line
length given a quasigroup must be a prime power (Lemma~\ref{pab}).  The
stronger (and much harder) result that no quasi-group can be interpreted in
any of the constructed strongly minimal Steiner $k$-systems is reported in
Fact~\ref{bvnobin}.

\begin{theorem}\label{disp2} If $T_\mu$ is the theory of a strongly minimal Steiner
    $k$-system $(M,R)$ constructed as in \cite{BaldwinPao}, then

\begin{enumerate}
\item If $k$ is a prime power $q$, for each Mikado $(2,q)$-variety $V$,
    there is a quasigroup in $V$ that interprets $(M,R)$ with lines as
    $2$-generated subalgebras and thus $T_\mu$ is interpreted in an
    incomplete theory $\check T_{\mu,V}$.
    \item If $k$ is not a prime power the Steiner system does have such
        an interpretation.
\item  Unless $q=3$, the interpreting quasigroup  is not interpretable in
    $(M,R)$.
\end{enumerate}
\end{theorem}

A complete first order theory is strongly minimal if every definable set in
every model is finite or cofinite.  Equivalently, in  a strongly minimal
theory $T$ the model theoretic notion of algebraic closure determines a
combinatorial geometry (matroid) with all bases automorphic. Model
theorists say $a \in \acl_M(B)$ if for some $\phi(x,\bbar)$ with $\bbar \in
B$, $M$ satisfies both $\phi(a,\bbar)$ and, for some $k$, $(\exists^{<k}
x)\phi(x,\bbar)$. We write bold face $\bbar$ to indicate a finite sequence
of elements, while $b$ is a singleton. Zilber
  conjectured that these geometries were all disintegrated ($\acl(A) =
\bigcup_{a\in A} \acl_M(a)$), locally modular (group-like), or field-like.
The examples here modify Hrushovki's construction that refuted this
conjecture \cite{Hrustrongmin}.  A geometry is \emph{ flat} \cite[Section
4.2]{Hrustrongmin} if the dimension of a closed subspace is determined from
its own closed subspaces by the inclusion-exclusion principle
\cite[Definition
 3.8]{BaldwinPao}. Hrushovski's flat counterexamples  have generally been
 regarded as an amorphous class of exotic structures.  Indeed, a
 distinguishing characteristic is the inability to formally define an
associative operation with infinite domain in any structure with a flat
$\acl$-geometry.  Here, we show that non-associative does not mean
uninteresting.

Although the Steiner systems  in Theorem~\ref{disp2}.(3) do not define a
quasi-group, when $q$ is a prime power we can find strongly minimal
quasigroups that induce strongly minimal Steiner $q$-systems.

\begin{theorem}\label{disp3}
For each $q$ and each of the $T_\mu$ in Theorem~\ref{disp2} with line
length $k = q =  p^n$ (for prime $p$) and each Mikado $(2,q)$ variety of
quasigroups $V$, there is a strongly minimal theory of quasigroups,
$T_{\mu',V}$ such that taking the $2$-generated sub-quasigroups as lines
yields a strongly minimal Steiner $q$-system.
\end{theorem}

We explain in Section~\ref{findqg} how $\mu'$ is generated from $\mu$ and
why the quasigroups satisfying $T_{\mu',V}$ are in $V$.

In \cite{BaldwinsmssIII} we investigate  various combinatorial problems
about the classes of quasigroups constructed here. In particular, we find
strongly minimal Steiner triple systems (whose automorphism groups are
two-transitive)  of every infinite cardinality, and then easily deduce they
have uniform cycle graphs \cite{CameronWebb}, and further that are
$\infty$-sparse in the sense of \cite{Chicoetal}.   We discus  in the
introduction and Remark 5.27 of \cite{BaldwinPao} and in
\cite{BaldwinsmssIII} the connections of this work with, among others,
Barbina-Casanovas, Conant-Kruckman, Horsley-Webb,  and Hyttinen-Paolini
\cite{BarbinaCasa, ConKr, HorsleyWebb, HyttinenPaolinifree}. These works
construct first order theories of Steiner systems or projective planes that
are at the other end of the stability spectrum from those here. Evans
\cite{Evanssteiner} uses the Hrushovski construction to address
combinatorial issues about Steiner systems.

This paper depends heavily on the results and notation of \cite{BaldwinPao,
BaldwinVer}. Certain arguments will require consulting those papers.

     \section{Model Theoretic Preliminaries}\label{bg}
     \numberwithin{theorem}{subsection}

In Section~\ref{gencon}, we lay out the general pattern of a construction
of a strongly minimal set by the `Hrushovki method'. Fra{\" \i}ss\'{e} and
J{\'o}nsson generalized  the Hausdorff notion of
 `universal' linear orders to  universally axiomatizable classes of
 structures. Hrushovski provided a `pre-processing' for this technique that
gives a general method for constructing theories of various model theoretic
complexities. We specify here various refinements of his method that apply
in universal algebra and combinatorics. We deal   primarily with `ab
initio' constructions that begin with a collection of finite structures as
opposed to expansions of structures (`bad fields') or fusions. The two page
 \cite[Section 2.1]{BaldwinPao} (arxiv) summarizes the role of strongly
minimal sets in model theory and how strongly minimal Steiner systems
arise.

 The basic ideas of the Hrushovski method are: i) Modify the
 Fra{\"\i}ss\'{e} construction of countable homogeneous-universal
 structures by replacing the relation of substructure between finite
structures  by a relation of strong substructure ($\leq$)  defined using a
pre-dimension function $\epsilon$ with $\epsilon(A) \in \NN$ for each
finite $A$ in a specified class $\bL_0$.

ii) Employ a function $\mu$ to bound the number $0$-primitive extensions of
each finite structure to obtain a class $(\bL_\mu,\leq)$. Then apply
\cite{Fraisse} to that class  so that closure in the geometry on the
generic model  is algebraic closure. Sections \ref{gencon} and
\ref{genspace} are a series of definitions and results needed to apply the
Hrushovski construction as modified in \cite{BaldwinPao}.

\subsection{The Hrushovski method}\label{gencon}

 In this section we first describe this method axiomatically while listing
 the five kinds of parameters that must be specified for any particular
 family of constructions. We slightly generalize Hrushovski's approach by
 using work of Kueker and Laskowski \cite{KuekerLas} to weaken  the
 requirement imposed by both  Fra{\" \i}ss\'{e} and the original Hrushovski
constructions that the collection of finite structures is closed under
substructure.

\begin{definition}[\cite{KuekerLas}]\label{smcl} \begin{enumerate}
\item  A countable collection $(\bL_0,\leq)$ of finite structures, with a
    transitive relation ($\leq$:  \emph{ strong substructure}) on $\bL_0$
    that refines substructure, is \emph{ smooth} if $B \leq C$ implies $B
    \leq C'$ if $B \subseteq C' \subseteq C$.

     \item Given a class of finite structures $L_0$, $\hat L_0$ denotes the
         collection of structures of direct limits of members of
$L_0$.
\end{enumerate}
\end{definition}

\begin{theorem}[\cite{KuekerLas}]\label{KL} If a smooth
    class satisfies the amalgamation and joint embedding properties there
    is a countable generic model $\Gscr$ (See Definition~\ref{gendef}.).

\end{theorem}

 The extension in \cite{KuekerLas} to an abstract treatment of a smooth
    class $(L,\leq)$ includes the Hrushovski construction of strongly
    minimal sets since both the definitions of the class and the strong
    extension relation are by universal sentences. To my knowledge, the
    construction here of strongly minimal quasigroups
    (Section~\ref{findqg}) is the first place where a $AE$-axiomatizable
    smooth class is used to study strongly minimal sets.

In \cite{Baldwinfg} we listed three of the major variants of the Hrushovski
method as of 2010 and that number has at least doubled in the ensuing
decade. Those variants range through the stability hierarchy and some
involve infinitary logics. The \emph{ fine structure of the original
method} \cite{Hrustrongmin} has been studied only in our recent work.

  We now describe the general framework for several different constructions
that appear in this paper; we use $(\sigma, \bL^*_0, \epsilon, \bL_0, {\bf
U})$ to make clear \emph{ this context holds throughout the paper}. We
replace $\bL, \epsilon$ by $\bK, \delta$ for the explicit cases and when
the possibility of confusion is even stronger add further labels.
Identifying the parameters of the method in Context~\ref{hruclass}
clarifies the relations among the variants as those parameters are
instantiated differently at several points in this
 paper as well as in \cite{BaldwinPao,
BaldwinVer, BaldwinsmssIII}.

\begin{context}\label{hruclass}{\rm A \emph{ Hrushovski sm-class} depends on
the choice of a quintuple $(\sigma, \bL^*_0, \epsilon, \bL_0, {\bf U})$ of
parameters.

\begin{enumerate}
\item The  vocabulary  $\sigma$ contains only relation  and constant
    symbols.

\item $\bL^*_0$ is a countable $\forall$-axiomatizable collection of
    finite $\sigma$-structures.

    \item A \emph{ pre-dimension} $\epsilon$ is a function from $\bL^*_0$
        to  the integers $\ZZ$ that, with $A \subseteq B$, writing
        $\epsilon(A/B)$ for $\epsilon(A) - \epsilon(B)$
satisfies:
  \begin{enumerate}

\item $\epsilon$ is submodular: That is, if $A, B, C \subseteq D \in
    \bL^{*}_0$, with $A \cap C = B$, then:
    \[\epsilon(A/B) \geq \epsilon(A/C),\] which an easy calculation shows
    is equivalent to submodularity: \[\epsilon(A\cup C) \geq
    \epsilon(A) + \epsilon(C) - \epsilon(B).\]
    \end{enumerate}

\item     $\bL_0$ is a subset of $\bL^*_0$ defined using $\epsilon$.
    Here, $\bL_0$ is those $A \in \bL^*_0$ such that    for any   subset
    $A'$ of $A$, $\epsilon(A') \geq 0$.

\end{enumerate}}
\end{context}

$\bU$ requires some preparation.

 Requirement (3) that $\epsilon$ maps into $\ZZ$ slightly weakens the
result in   Baldwin and Shi \cite{BaldwinShiJapan} that  well-ordering of
the range suffices to get an $\omega$-stable generic model with  a geometry
rather than just a dependence notion. They show that by allowing real
coefficients one obtains a stable theory with the forking relation as
dependence. From such an $\epsilon$, one defines notions of strong
extension ($\leq$), primitive extension, and good pair.

\begin{definition}[Strong Extensions]\label{stext}
\begin{enumerate}
\item In any $N \in \hat\bL_0$, $$d_N(A/B) = \min\{\epsilon(A'/B): B
    \subseteq A \subseteq A' \subseteq N\}.$$

     We often write $d_N(A-B/B)$ for $d_n(A/B)$.

    \[d_N(A) = d_N(A/\emptyset).\]
    \item For any $N \in \hat\bL_0$ with $B \subseteq N$, we write $B
        \leq N$ (read $N$ is a \emph{ strong extension} of $B$) when  $B
        \subseteq A\subseteq N$ implies $d_N(A) \geq d_N(B)$.

\item We write $B < A$ to mean that $B \leq A$ and $B$ is a proper subset
    of $A$.

\end{enumerate}
\end{definition}

The following definitions describe the pairs $B \subseteq A$ such that, in
the generic model constructed from the class $\bL_0$, $A$ will be contained
in the algebraic closure of $B$.  \emph{ We  write $(A-B/B)$ or $(A/B)$ for
the same pair depending on whether the superset or the annulus is
emphasized.}

\begin{definition}[Primitive and Good]\label{prealgebraic} Let $B, C \in \bL_0$ with $B \cap C= \emptyset$ and $C \neq \emptyset$.
Write $A$ for $B \cup C$.
	\begin{enumerate}
\item $A$ is a \emph{ $k$-primitive extension} of/over $B$ if $B \leq A$,
    $\epsilon(A/B) =k$, and there is no $A_0$ with $B \subsetneq A_0
    \subsetneq A$ such that $B \leq A_0 \leq A$. We may just write \emph{
    primitive} when $k=0$.

 We stress that in this definition, while $B$ may be empty, $A$ cannot
 be.
	\item  We say that the $0$-primitive extension $A/B$ is \emph{ good} if
there is no $B' \subsetneq B$ such that  $(A/B')$ is $0$-primitive.
(Hrushovski called this a minimal simply algebraic or m.s.a.
extension.)

	\item If $A$ is $0$-primitive over $B$ and $B' \subseteq B$ is such
that we have that $A/B'$ is good, then we say that $B'$ is a \emph{ base}
for $A$

	\item If  $A/B$ is good, then we also write $(A/B)$ is a \emph{ good pair}.
\end{enumerate}
\end{definition}

Hrushovski gave one technical condition on the function $\mu$ counting the
number of realizations of a good pair that ensured the theory is strongly
minimal rather than $\omega$-stable of rank $\omega$.  Fixing a class ${\bf
U}$ of functions $\mu$ satisfying that condition in the base case and other
conditions for special purposes provides a way to index a rich group of
distinct constructions. At various times in this paper ${\bf U}$ is
instantiated as $\Uscr$, $\Uscr'$, $\Cscr$,  or $\Tscr$.

	\begin{definition}[$\mu$ and $U$]    \label{Kmu}	We describe the
functions that impose algebraicity.
	\begin{enumerate}

	\item\label{itemLmu} Let $\bU$ be  collection of functions $\mu$
assigning to every isomorphism type $\boldsymbol{\beta}$ of a good pair
$(C/B)$ in $\bL_0$ a non-negative integer. her good pairs may be allowed
at least twice as many realizations

\item For any good pair $(C/B)$ with $B \subseteq M$ and $M \in \hat
    \bL_0$, $\chi_M(C/B)$ denotes the number of disjoint copies of $C$
 over $B$ in $M$. Of course,  $\chi_M(C/B)$ may be $0$.

	\item\label{Kmuitem} For any $\mu \in \bU$,  $\bL_{\mu}$ is the class
of structures $M$ in $ \bL_{0}$ such that if $(C/B)$ is a good pair
$\chi_M(C/B)  \leq \mu(C/B)$.

\end{enumerate}	
\end{definition}

\begin{definition}\label{gendef} A  countable structure $\Gscr_\mu$ is \emph{ generic} for $\bL_\mu$ if
\begin{enumerate}
\item it is a countable union of structures in $\bL_\mu$;
\item and it is $\leq$-homogenous: if  isomorphic finite $A,B$ are each
    strong in $\Gscr_\mu$, they are automorphic in $\Gscr_\mu$.
\end{enumerate}
\end{definition}

\begin{theorem}
\label{Hrugen} If  $(\bL_\mu,\leq)$ is a smooth class with the amalgamation
property then it has a countable generic model $\Gscr_\mu$.
\end{theorem}

Smoothness is immediate for Hrushovski and for \cite{BaldwinPao} as $L_\mu$
and $\leq$ are given by universal sentences; Section~\ref{findqg} requires
more care (Theorem~\ref{getsmquasigrp}) because the class of finite
structures is $\forall\exists$ axiomatizable. Proofs that the theory of
$\Gscr_\mu$ is strongly minimal depend slightly on the particular instance
of the schema described in this section; several such instantiations are
explained in Sections~\ref{genspace} and \ref{findqg}.

\begin{notation}
    The theory of the generic structure, $\Gscr_\mu$, is the desired
    strongly minimal theory $T_\mu$.
\end{notation}

\subsection{Generic Linear Spaces}\label{genspace}

The construction of strongly minimal Steiner systems \cite{BaldwinPao}
takes place with the  instantiation in Definition~\ref{BPsetting} for
linear spaces of the pattern described in Context~\ref{hruclass}.

Formalizing the initial description,

 \begin{definition} A linear space is a
structure $(M,R)$ with a single ternary relation that is set-like (holds
only of distinct tuples and in any order or none) and one further axiom:
two points determine a line.
\end{definition}

\begin{definition}\label{lines}  We say a \emph{ maximal} $R$-clique in a
linear space $M$ is a line (block) and sometimes   write (partial) line for
a clique that is not maximal.  Note that if $B \subset A \subseteq M$ then
a maximal clique in $B$ may not be maximal in $A$.
 \begin{enumerate}

 \item Two unrelated points in a linear space $(M,R)$ are regarded as
     being on a trivial line. A non-trivial line is any $R$-clique of at
     least $3$-points. In a $k$-Steiner system, every line has $k>2$
     points and so is non-trivial.

 \item  For a line (maximal clique) $\ell \subseteq B$,  we denote the
     cardinality of a line $\ell$ by $|\ell|$, and, for $A \subseteq B$,
     we denote by $|\ell|_A$ the cardinality of $\ell \cap A$.
   \item We say that a non-trivial line $\ell$ contained in $B$ is \emph{
       based in} $A \subseteq B$ if $|\ell\cap A| \geq 2$, in this case
       we write $\ell \in L(B)$.
\item	\label{nullity} The \emph{ nullity of a line} $\ell$ contained in
    a linear space   $A$ is:
\[\mathbf{n}_A(\ell) = |\ell| - 2.\]
\end{enumerate}

 \end{definition}

  We deduce the notions of $d_N, \leq$, primitive,  and good,    exactly as
in Section~\ref{gencon} from the following specification of $\sigma,
\bL^*_0,\delta$.

\begin{definition}\label{BPsetting}
\begin{enumerate}
\item $\sigma \rightarrow \tau$: $\tau$ has a single ternary relation,
    $R$.

\item $\bL^*_0 \rightarrow \bK^*_0$:  $\bK^*_0$ is the class of finite
    linear spaces. In particular,  $R$ can hold only of three distinct
    elements and then in any order (i.e., is a $3$-hypergraph).

    \item $\epsilon \rightarrow \delta$: $\delta(A) = |A| - \sum_{\ell
        \in L(A)} \mathbf{n}_A(\ell)$ where $\mathbf{n}_A$ is defined in
Definition~\ref{Kmu}.

\item $\bL_0 \rightarrow \bK_0$: $\bK_0 = \{ A \in \bK^*_{0} \text{ such
    that for any } A' \subseteq A, \delta(A') \geq 0\}$.
\item $\bU \rightarrow \Uscr$:

\begin{enumerate} \item We write $\boldsymbol{\alpha}$ for the isomorphism type of a
     pair of sets $(\{b_1,b_2\},\{a\})$ with $R(b_1,b_2,a)$. (That is,
     rather than repeating the elements of the base by writing ($\{a,
     b_1,b_2\}/\{b_1,b_2\}$), we simply separate the two pieces of the
     diagram of the larger set.) $(\{b_1,b_2\},\{a\})$ will be a good
     pair in each example considered.
    \item\label{itemKmu} Let $\Uscr$ be the collection of functions $\mu$
        assigning to every isomorphism type $\boldsymbol{\beta}$ of a
        good pair $C/B$ in $\bL_0$:
	\begin{enumerate}
	\item a natural number $\mu(\boldsymbol{\beta}) = \mu(C/B) \geq \epsilon(B)$, if $|C-B|\geq 2$;
	\item a natural number $\mu(\boldsymbol{\beta}) \geq 1$, if $\boldsymbol{\beta} = \boldsymbol{\alpha}$.
	\end{enumerate}

 \end{enumerate}
 \end{enumerate}
\end{definition}

The special treatment of $\boldsymbol{\alpha}$ is to allow the
consideration of Steiner $3$-systems. Note that in
Definition~\ref{BPsetting}, the class $\bK^*_0$ is $\forall$-axiomatizable;
in Section~\ref{findqg}, we will need a $\forall\exists$ class for the
relevant instantiation $\tilde \bK_q$ of $\bL_0$.

\cite[Lemma 3.10.3]{BaldwinPao} demonstrates the class $\bK_0$ satisfies
amalgamation with the following construction of the canonical amalgam.

\begin{definition}\cite[Definition 3.7]{BaldwinPao}\label{defcanam} Let $A \cap B =C$
 with $A,B,C \in \bK_0$.
We define $D := A \oplus_{C} B$ as follows:

\begin{enumerate}
	\item the domain of $D$ is $A \cup B$;
	\item a pair of points  $a \in A - C$  and $b \in B - C$ are on a
non-trivial line $\ell'$ in $D$ if and only if there is line $\ell$ based
in $C$ such that $a \in \ell$ (in $A$) and $b \in \ell$ (in $B$). Thus
$\ell'=\ell$ (in $D$).
\end{enumerate}
\end{definition}

Baldwin and Paolini \cite{BaldwinPao} demonstrate  the class
$(\bK_0,\delta)$ satisfies the basic properties (including flatness) of a
$\delta$ function in a Hrushovski construction and of the associated
algebraic closure geometry. For $M \models T_\mu$, $a\in \acl(B)$ if and
only $d_M(a/B) =0$. The flatness implies that no model of $T_\mu$
(Fact~\ref{disp2a}) admits a definable binary associative function with
infinite domain (\cite[Lemma 14]{Hrustrongmin}).

 The following lemma singles out the effect of the fact that our  $\delta$
  (Definition~\ref{BPsetting}.(3)) depends on line length rather than the
  number of number of tuples realizing  $R$.

\begin{fact}[Line length]\label{linelength}
By Lemma 5.18

of \cite{BaldwinPao}, lines in models of $T_\mu$ have length $k$ if and only
if $\mu(\boldsymbol{\alpha}) = k-2$.
 \end{fact}

\subsection{Interpretations } \label{coord}

  We carefully define the concept of interpretation as given in
\cite{Hodgesbook}. To bridge the several fields considered  here we write
these definitions in the notation of this paper. In Section~\ref{assqg} the
Ganter-Werner notion of `coordinatizing' Steiner systems by quasigroups is
seen as a specific kind of interpretation.

By a \emph{ vocabulary} (alias: similarity type, language, signature) we
mean a list of function and relation symbols.  A boldfaced variable
represents a finite sequence  of variables. A formula is \emph{ unnested}
if any atomic subformula  of it is either a single occurrence of a relation
symbol or
        an equality of terms containing only variables and at most one
        function symbol \cite[p 58]{Hodgesbook}.

\begin{definition}[Interpretations]\label{interpdef} Fix vocabularies $\tau$ and
$\sigma$. Fix also a $\tau$-structure $B$, a $\sigma$-structure $A$ and a
positive integer $n$. An $n$-\emph{
 dimensional} interpretation $\Gamma$ of $B$ into $A$ is:

\begin{enumerate}
\item
\begin{enumerate}
\item a $\sigma$-formula $\partial_{\Gamma}(x_0 \ldots x_{n-1})$ (the
    domain of the interpreted model).
  \item For each unnested atomic $\tau$-formula $\phi(y_0, \ldots
      y_{m-1})$,
       a $\sigma$-formula $\phi_\Gamma(\xbar)$ where $\xbar$ is an
      $m$-tuple of $n$-tuples.

        \item there is a surjective map $f_\Gamma \colon
            \partial_{\Gamma}(A^n) \rightarrow B$ such that for each
            unnested atomic $\tau$-formula $\phi(\ybar)$ and
          any $\abar_i \in \partial_{\Gamma}(A^n)$

            \[B \models \phi(f_\Gamma(\abar_0), \ldots
            f_\Gamma(\abar_{m-1}) ) \leftrightarrow A \models
            \phi(\abar_0, \ldots \abar_{m-1})\]
   \end{enumerate}
\item Note that 1a) and 1b) have established a function from the
    vocabulary of $\tau$ to formulas of $\sigma$; this extends by
    inductions on formulas to arbitrary formulas (See \cite[5.3.2,Remark
    1]{Hodgesbook}.  Such a map is called an interpretation of $\tau$
    into $\sigma$.

   \item An interpretation $\Gamma$ of a class $\mathfrak{S}$ of
       $\tau$-structures into a class $V$ of $\sigma   $-structures (or
       of the theories of such classes) is just an interpretation of
       $\tau$ into $\sigma$  \cite[(b) p 221]{Hodgesbook}. This rather
       weak correspondence becomes useful when various properties in
       Definition~\ref{interpprop} hold. .
\end{enumerate}
\end{definition}
  In this paper theory means first order theory; we will sometimes specify
        the type of theory with such terms as variety, complete, or
        strongly minimal.

The natural admissibility conditions (expressed by a set of first order
formulas) assumed in Lemma~\ref{funct}.(2) are detailed at \cite[page
214]{Hodgesbook}; they allow for taking quotients in performing an
interpretation and yield the following result \cite[5.3.2,
5.3.4]{Hodgesbook}. We write   $\approx$ for isomorphism in the appropriate
        vocabulary.

\begin{lemma}\cite[Theorem 5.5.3]{Hodgesbook} \label{funct}
Suppose $\Gamma$ is  an $n$-\emph{ dimensional} interpretation of the
$\tau$-structure $B$ into the  $\sigma$-structure $A$.
\begin{enumerate}
\item Definition~\ref{interpdef}.(1).(c) holds for \emph{ all}
    $\tau$-formulas $\phi$.
    \item For every $\sigma$-structure $A$ which satisfies the
        admissibility conditions there is a $\tau$-structure $B$ and a
        map $f=f_\Gamma$ with $f_\Gamma\colon\partial_{\Gamma}(A)
        \rightarrow B$ such that:

        \begin{enumerate}
\item If $g$ and $C$ are such that $\Gamma$ with
    $g\colon\partial_\Gamma(A) \rightarrow$ is also an interpretation
    of $C$ into $A$ then there is an isomorphism $i\colon B \rightarrow
    C$ such that $i(f\abar) = g(\abar)$ for every $\abar \in
    \partial_{\Gamma}(A^n)$.
    \item We write $\Gamma A$ for the isomorphism class of  $B \approx
        C$.
    \end{enumerate}
    \end{enumerate}
    \end{lemma}

The interpretations constructed in this paper will be $1$-dimensional,
indeed on the same domain. But, Theorem~\ref{bvnobin} shows not even
$n$-dimensional interpretations are possible in the other direction. Hodges
introduced the following terminology to detail which additional properties
of an interpretation were important for applications to decidability,
consistency, model theoretic complexity, etc. In particular they will allow
us in Section~\ref{assqg} to clarify the strength and weaknesses of
coordinatization.

\begin{definition}[Properties of interpretations]\label{interpprop} Let $\Gamma$ be an
       interpretation of a class $\mathfrak{S}$ of $\tau$-structures into a
       class $V$ of $\sigma$-structures.
\begin{enumerate}
\item $\Gamma$ is \emph{ left total} if for every $A \in V$, $\Gamma A\in
    \mathfrak{S}$.
\item $\Gamma$ is \emph{ right total} if for every $B \in \mathfrak{S}$,
    there is an $ A\in V$ with $\Gamma A \approx B$.
    \item $\Gamma$ is \emph{   total} if it is both  left and right
        total.

     \end{enumerate}

\end{definition}

\begin{definition}[
     Relations of Classes]\label{mutbi}

\begin{enumerate}
\item Two classes of structures are mutually interpretable if each is
    interpreted in the other. \item They are bi-interpretable if in
    addition the composition of the two interpretations (in either
    direction) is the identity.
 \end{enumerate}

\end{definition}

\begin{remark}[One sorted formalization]\label{1sort}  Since each sort is infinite,
no theory in  the two-sorted formulation of `linear space' can  be strongly
minimal. However, we showed in Section 2 of \cite{BaldwinPao} that  there
is a ($2$-dimensional)  left and right total bi-interpretation between
linear spaces in the two sorted formalization and linear spaces in a
one-sorted logic  with a single ternary `collinearity' predicate.
Fact~\ref{disp2}.1 applied the Hrushovski method to linear spaces in the
one-sorted framework and using a geometrically motivated predimension
function (Definition~\ref{BPsetting}.(3)) produced strongly minimal Steiner
$k$-systems that are model complete and satisfy the usual properties of
counterexamples to Zilber's trichotomy
 conjecture. Their $\acl$-geometries are flat, but  not disintegrated  nor
 locally modular.

 \end{remark}

 \section{Associating  Strongly Minimal Steiner systems with
 Quasigroups}\label{assqg} \numberwithin{theorem}{section}

We summarise and extend the substantial literature on coordinatization of
$k$-Steiner systems to show $T_\mu$ is interpretable into a theory $\check
T_{\mu,V}$ of quasigroups. We note in Lemma~\ref{defsq} that Steiner
$3$-systems are quasigroups. Then we give a short proof that the `natural'
coordinatizing quasigroup provided by \cite{GanWer} is not definable in the
strongly minimal $k$-Steiner system $(M,R)$ when $k> 3$. Thus, one can't
`invert' the coordinatization to obtain an interpretation of the quasigroup
into the Steiner system.

 Then we  deduce from the argument of Baldwin-Verbovskiy \cite{BaldwinVer}
Theorem~\ref{bvnobin} that under a weak hypothesis ($\mu$-triples,
Definition~\ref{tripdef}) no binary function with domain $M^2$ is
interpretable  in any $(M,R) \models T_\mu$.

 \begin{notation}\label{Steinersysnot}
For each $k \in \omega$, $\mathfrak{S}^k$ denotes the class   of Steiner
$k$-systems. If we write $q$, we mean a prime power.
\end{notation}

\cite{GanWer} use the notation $\mathfrak{S}^m_k$; for them $m$ is block
size and $k$ is the number of points that determine a line. By restricting
to linear spaces we have fixed that $k$ as $2$ and need only a single
parameter, which we superscript.

\begin{definition}[\cite{Smithlec}]\label{quasigroups}
\begin{enumerate}
\item A structure $(A,*)$ with one binary function $*$ is called a \emph{
    groupoid }(or \emph{ magma}).

\item A \emph{ quasigroup} (Alias:  multiplicative quasigroup
    \cite{McNetal}, combinatorial quasigroup \cite{Smithlec})   $(Q,*)$
    is a groupoid $(Q,*)$ such that for $a,b\in  Q$, there exist unique
    elements $x,y
\in Q$ such that both \[a *x = b \ \text{and} \ y * a = b.\]

\item If every $2$-generated subquasigroup has $q$-elements it is a
    \emph{ $q$-quasigroup}.

\end{enumerate}
\end{definition}

 The general notion of a quasigroup  is an $AE$ Horn class in the
     vocabulary with function symbol $*$. So in general, a  quotient or a
subalgebra of a quasigroup $(Q,*)$ need not be a quasigroup.  But
Quackenbush provides a sufficient condition that every algebra in the
variety generated by $Q$ is a quasigroup.

\begin{theorem} \cite[Theorem 3]{Quackquac} If $(Q,*)$ is a quasigroup  with $V =(HSP(Q))$ and
$F_2(V)$, the free $V$-algebra on 2 generators, is finite then every
algebra in $V$ is a quasigroup.
 \end{theorem}

 \begin{remark}\label{vocab} {\rm Quackenbush's argument makes an interesting use of the
 finiteness of $F_2(V)$. By standard arguments (forbidden footnote: Recall
 the operators $H,S,P$ taking a class of algebras $\bK$ to its homomorphic
 images, subalgebras, and direct products respectively.) $F_2(HSP(Q))$ is
 in $SP(Q)$ so it satisfies the cancellation laws. Thus, left and right
 multiplication are injective. But since injective maps of finite sets are
 onto, each equation $ax=b$ has a solution and $F_2$ is a quasigroup. So we
 can treat the quasigroups that arise as structures with one binary
operation and deal only with varieties.  This allows us to apply directly
the results of \cite{BaldwinPao}.}
 \end{remark}

  We discuss in detail three (families of) varieties of quasigroups
corresponding to $k=1,2, p^n$.

\begin{definition}\label{defrelvar}\cite{Smithlec}

 \begin{enumerate}
\item A \emph{ Steiner quasigroup} is a groupoid which satisfies the
    equations: $x \circ x = x, x \circ y = y \circ x, x \circ (x \circ y)
    = y $.
     \item A \emph{ Stein quasigroup}

         is a  groupoid which satisfies the  equations: $x*x =x$,
         $(x*y)*y = y*x$, $(y*x)*y = x$.

\item \emph{block algebras}: \cite{GanWer2} Let $q=p^n$ for some prime
    $p$ and natural number $n$.
 \begin{enumerate}
 \item A near-field is an algebraic structure in a vocabulary
     $(+,\times, 0,1)$
    satisfying the axioms for a division ring, except that it has only
    one of the two distributive laws.

 \item Given a near-field $(F,+,\cdot,-,0,1)$ of cardinality $q$ and a
     primitive element $a \in F$, define a multiplication $*$ on $F$ by
     $x*y = y+ (x-y)a$. An algebra $(A,*)$ satisfying the 2-variable
     identities of $(F,*)$ is a \emph{ block algebra} \cite{GanWer2}
     over $(F,*)$; $(F,*)$ is idempotent ($x*x =x$).

            \end{enumerate}

            \end{enumerate}
    \end{definition}

  While every group is a quasigroup, the  Stein and Steiner quasigroups are
rather special quasigroups since they are idempotent. Thus, a Stein  or
Steiner quasigroup $(Q,*)$ cannot be a group unless it has only one
element. Further, it routine to check from the defining equations that
there is a \emph{ unique} (up to isomorphism) $2$-generated Steiner (Stein)
quasigroup and it has $3$ ($4$) elements. So it is necessarily both simple
and free. Block algebras are quasigroups but when $q>4$ neither Stein nor
Steiner quasigroups.

 \begin{lemma}\label{defsq}

Each  Steiner \emph{ triple} system is bi-interpretable with a Steiner
    quasigroup (Definition~\ref{defrelvar}).

\end{lemma}

\begin{proof}  Given the algebra, the lines are the 2-generated subalgebras,
which are easily seen from the defining equations to  have cardinality $3$.
Given a Steiner triple system, let $x \circ y$ be the third element of the
line if $x \neq y$ and $x \circ x = x$.  Since all lines are isomorphic to
the unique 3 element Steiner quasigroup, the resulting algebra is a Steiner
quasigroup. \end{proof}

In general, \cite{BaldwinPao} gave us a theory $T_\mu$ of Steiner
 $q$-systems for each prime power $q$; when $\mu(\boldsymbol{\alpha}) =1$
 we get a $3$-Steiner system and hence also a quasigroup.

\begin{corollary}\label{manyqg} For each prime power $q$, there are $2^{\aleph_0}$
strongly minimal theories $T_\mu$ of Steiner $q$-systems and so,  when
 $\mu(\boldsymbol{\alpha}) =1$, non-isomorphic (and even not elementarily
 equivalent) quasigroups of cardinality $\aleph_0$.
\end{corollary}

\begin{proof}  Lemma~\ref{defsq} provides an explicit $1$-dimensional (the domain and range of
the interpretation is the universe) bi-interpretation between Steiner
triple systems and the Steiner $3$-systems from Theorem~\ref{disp2}.1
\cite[Corollary 5.23]{BaldwinPao}.

\end{proof}

We examine the notion of coordinatization from Ganter and Werner
\cite{GanWer, GanWer2} and describe the methods of their proof of the
coordinatization of $\mathfrak{S}^q$ by varieties of quasigroups.  We then
adapt these methods in the remainder of the paper to analyze the possible
interpretations between strongly minimal Steiner systems and quasigroups.
We begin with a precise definition of `coordinatization'.

\begin{definition}[Coordinatization]\label{coordef}
\begin{enumerate}
 \cite{GanWer, GanWer2} \item  A Steiner $q$-system $(M,R)$

is \emph{ coordinatized by a quasigroup}  $(M,*)$ if the lines of $(M,R)$
are the $2$-generated subalgebras of $(M,*)$.
\item The class $\mathfrak{S}^q$ of Steiner systems is \emph{
    coordinatized by a variety} $V$ if each $(M,R) \in \mathfrak{S}^q$ is
    coordinatized by an $(M,*)\in V$.
\end{enumerate}

\end{definition}

The proof of Lemma~\ref{Steinq} motivates the generalization in
Lemma~\ref{directgammaA}.

 \begin{lemma}
 [Stein quasigroups \& Steiner $4$ systems] \cite[page 5]{GanWer2}
\label{Steinq} Each Stein quasigroup induces a Steiner  $4$-system $(M,R)$.
Moreover, each Steiner  $4$-system $(M,R)$ is coordinatized by a Stein
quasigroup, $(M,*)$.
\end{lemma}

\begin{proof}  One direction is obvious; the lines are the $2$-generated
subalgebras of the quasigroup, which, as noted after
Definition~\ref{defrelvar}, all have cardinality $4$. For the other
direction, the universe of the algebra is $M$. We noted above just before
Lemma~\ref{defsq}, that all Stein  $4$-quasigroups are isomorphic and
strictly $2$-transitive. So, if we arbitrarily impose the structure of a
Stein  $4$-quasigroup on each line the entire structure is a Stein
quasigroup. It clearly satisfies the three equations of
Definition~\ref{defrelvar}.2 because they involve elements only within a
single block and also the requirement that each equation $ax =b$ ($ya =b$)
has a unique solution, as again the solution is within the block determined
by $a,b$. \end{proof}

\medskip

Coordinatization as in Definition~\ref{coordef} does not necessarily give
an interpretation. It does, if the variety $V$ satisfies the stronger
properties of an $(r,k)$-variety which we now give. \cite{GanWer} used
these conditions to extend the
 coordinatization phenomena to Steiner $q$-systems for prime power $q$.

 \begin{definition}[\cite{GanWer, Padchar}]\label{rkdef}

  \begin{enumerate} \item The variety $V$ is an $(r,k)$-variety if every $r$-generated
 subalgebra of any $A \in V$ is isomorphic to $F_r(V)$, the free
 $V$-algebra on $r$ generators, and $|F_r(V)|=k$.

 \item  A variety $V$ is \emph{ binary} \cite{Evans2var} if both all
     function symbols of $V$ are binary and the defining equations
     involve only 2 variables.

 \item A \emph{ $q$-Mikado} variety  \cite[p. 129]{GanWer} is a binary
     $(2,q)$-variety.
\end{enumerate}
 \end{definition}

Once the statements are understood, the proofs of the following
equivalences are straightforward.

\begin{lemma}[\cite{Padchar}]\label{pab} For a variety $V$ of groupoids, the following are equivalent:
\begin{enumerate}
\item $V$ is a $(2,k)$-variety;
\item Taking the $2$-generated subalgebras of any $A \in V$ as lines
    yields a Steiner $k$-system;
    \item the automorphism group of any $2$-generated algebra is strictly
        (i.e. sharply) $2$-transitive.
\end{enumerate}
\end{lemma}

Abusing notation, we say that an algebra  $A$ satisfying
Lemma~\ref{pab}.(3) is strictly $2$-transitive. And, as \cite{Swierfreealg}
points out, applying the strict two transitivity to $F_2(V)$, an argument
of Burnside \cite{Burnside}, \cite[Theorem 7.3.1]{Robinson1982} shows:

\begin{corollary}\label{kq} If $V$ is a $(2,k)$-variety then $k$ is a prime power.
\end{corollary}

The salient characteristic  (crucial for e.g. Lemma~\ref{2propagates}) of
the equational theories that arise in this paper is that each defining
equation involves only two variables. In particular, none of the varieties
are associative. We rely heavily on a `classical' observation of Trevor
Evans. It requires no written proof, but a little thought.

\begin{lemma}[\cite{Steinq,Evansmonthly}] \label{2propagates} If $V$ is a  binary variety
 of idempotent algebras and each line of a Steiner system $(M,R)$ is
expanded to  an algebra from $V$ then the resulting algebra is in $V$.
\end{lemma}

\begin{theorem}\label{directgammaA}
For any Mikado variety $V$ of $q$-quasigroups and any $(M,*) \in V$, the
Definition~\ref{coordef} coordinatization of a Steiner system $(M,R)$ is an
interpretation of $(M,R)$ into $(M,*)$ and thus yields a left and right
total interpretation $\mathfrak{S}^q$ into $V$.
\end{theorem}

\begin{proof} Left total:
$\Gamma_V((M,*))$  is the $\tau$-structure $(M,R)$ with universe $M$ where
$R$ is defined as follows.

Let $\theta_F(x,y, z)$ be the disjunction of the  atomic formulas $z =
f_i(x,y)$ where the
      $f_i(x,y)$ list the terms generating $F=F_2(V)$ from $\{x,y\}$. Then
      letting $R(x,y,z) \leftrightarrow \theta_F(x,y,z)$
defines   a relation $R$ such that $(M,R)$ is a Steiner $q$-system
consisting of the $2$-generated subalgebras of $(M,*)$.

Right total: Conversely, let $(M,R) \in \mathfrak{S}^q$. By
Lemma~\ref{2propagates}, expanding $(M,R)$ by placing an arbitrary copy of
$F_2(V)$ as multiplication on each line gives an algebra $(M,*)$ in $V$
with $\Gamma_V(M,*) \approx (M,R)$.\end{proof}

The argument so far, with different terminology,  is in \cite{GanWer}.  We
now explore the restrictions on the interpreting quasigroup  if the Steiner
system is required to satisfy a $T_\mu$.  As Ganter and Werner \cite[p
7]{GanWer2} point out, this interpretation into quasigroups is not unique.
They describe two different varieties of block algebras (one commutative
and one not) over $F_5$, depending on the choice of the primitive element
$a$ of $F_5$ (Definition~\ref{defrelvar}). Either can be used to
coordinatize a Steiner $5$-system. Thus the theory of the Steiner system
$(M,R)$ does not even predict the
  equational theory of the
 coordinatizing algebra and certainly does not control the first order
  theory. That is why we label the interpreting theory in
  Theorem~\ref{coordthm}: $\check T_{\mu, V}$.  Theorem~\ref{coordthm}.(2)
  proves the coordinatizing multiplication is not defined in the Steiner
System; so, we cannot invert \emph{ this} interpretation.

In Lemma~\ref{bvnobin}, we show the stronger result that there is no
interpretation of any sort of a quasigroup in a model of a $T_\mu$.

 \begin{theorem}\label{coordthm} If $T_\mu$ is the theory of a strongly
  minimal Steiner $q$-system
  (from Theorem~\ref{disp2}.1) and  $V$ is a Mikado $(2,q)$ variety of
  quasigroups, then

  \begin{enumerate}

\item There is a left and right total interpretation of $T_\mu$, a
    complete
 theory of Steiner $q$-systems, into $\check T_{\mu,V}$, an incomplete
 theory of quasigroups.

 \item If $\mu(\boldsymbol{\alpha})= q-2 >1$, the multiplication \emph{
     given by the coordinatization} is not definable in $(M,R)$.
 \end{enumerate}

\end{theorem}

\begin{proof}
1) We verify that the coordinatization of Theorem~\ref{directgammaA} is an
interpretation. Let $\delta_F(x,y, f_1(x,y), \ldots f_k(x,y))$ denote the
quantifier-free diagram of $F=F_2(V)$.  The strict $2$-transitivity of $F$
guarantees the particular choice of the two elements $x,y$ does not matter.
Letting $R/\delta_F$ denote the substitution of $\delta_F$ for $R$, $\check
T_{\mu,V}$ is axiomatized by \[Eq(V) \cup \{(\forall x,y) \delta_F(x,y,
f_1(x,y), \ldots f_k(x,y))\} \cup \{\phi\myrestriction (R/\delta_F): \phi
\in T_\mu\}.\] Left and right total are as proved in
Theorem~\ref{directgammaA}.

 2) Without loss of generality, let $(M,R)$  be the countable generic for
 $T_\mu$ and suppose it is coordinatized by $(M,*)$. In a linear space $A$,
 two points $a,b$ such that no $c\in A$ satisfies $R(a,b,c)$ satisfy
$d_A(\{a,b\}) =2$. Let $\{a,b\}$ be a
 strong substructure of $(M,R)$ (i.e. $d_M(\{a,b\} =2$; see
Definition~\ref{stext}.(1)) and let $c_1, \dots c_k$ fill out the line
  through $a,b$ to a structure $A$.  By genericity there is a strong
  embedding of $A$ into $M$. Since $V$ is a Mikado variety, all triples
$a,b,c_i$ realize the same quantifier free $R$-type and $A\leq M$ implies
  for any permutation  $\nu$ of $k$ fixing $0,1$,   for $2\leq i < k$,
  there is an automorphism of $(M,R)$ fixing $a,b$ and taking $c_i$ to
  $c_{\nu(i)}$. Thus, $a*b$ cannot be definable in $(M,R)$.
 \end{proof}

We have found an interpretation of Steiner $q$-systems in $q$-quasigroups.
We showed Theorem~\ref{coordthm}.(2) that the coordinatization does not
yield an interpretation in the other direction. But as we sketch in
\ref{tripdef}-\ref{nounary}, with a minor (triples) hypothesis on $\mu$,
Definition~\ref{tripdef},  there is no interpretation of any dimension of
any quasi-group into $T_\mu$,  even if we allow a first order definition of
the image of $*$. In fact, the multiplication $*$   cannot be defined in
$T_\mu$; as, there are no non-trivial definable binary functions on models
of $T_\mu$.

\begin{definition}[Triples]\label{tripdef} Let $\bK^*_0$ be the class of
finite linear spaces as in Definition~\ref{lines}. Define $\Tscr$ as the
collection of functions  $\mu$ from good pairs into $\omega$ such that
$\mu\in \Uscr$ (Definition~\ref{BPsetting}.{\ref{itemKmu}})  and such that

       for every good pair $(A/B)$ with $|A| > 1$ and $\delta(B) =2$,
       $\mu(A/B) \geq 3$.  We will write `$\mu$ triples', \cite[Theorem
       5.2]{BaldwinVer}.
        \end{definition}

        We say `triples' is a weak condition because it addresses only good
        pairs $A/B$ where   $B$ has small dimension. If $\mu$ is in the
class $\Tscr$ of triplable $\mu$-functions, \cite{BaldwinVer} ensures that
there are no definable truly binary   functions in the following precise
sense; several equivalents are given by \cite[Lemma 2.10]{BaldwinVer}.  The
following notion generalizes Gratzer's notion of function being
distinguished by one of its variables \cite[p 201]{Gratzer}.

         \begin{definition}[Essentially Unary functions]\label{essunary} Let $T$
be a strongly minimal theory. An $\emptyset$-definable function $f(x_0
    \ldots x_{n-1})$ is called \emph{ essentially unary} if there is an
$\emptyset$-definable function $g(u)$ such that for some $i$, for all but a
finite number of $c \in M$, and all but  a set of Morley rank $ <n$ of
tuples $\bbar\in M^n$, $f(b_0, \ldots b_{i-1}, c, b_{i+1}, \ldots b_{n-1})
= g(c)$.
        \end{definition}

 With Verbovskiy, we introduced the notion of a decomposition of finite
$G$-normal subsets   \cite[\S 2]{BaldwinVer} of Hrushovski strongly minimal
sets with respect to automorphism groups  to prove:

\begin{fact}\label{bvnobin}\cite[Theorem 5.6]{BaldwinVer} For any strongly minimal Steiner system $(M,R)$
\begin{enumerate}
\item If $\mu \in \Tscr$ ($\mu$-triples),  every definable function  in a
    model of $T_\mu$ is essentially unary.

\item If $\mu \in \Uscr$, $T_\mu$ does not have any commutative definable
    binary function.
    \end{enumerate}
\end{fact}

   Fact~\ref{bvnobin}     is proved for the basic Hrushovski construction
   and strongly minimal Steiner systems in \cite{BaldwinVer}.  The crucial
   distinction between (1) and (2) in Fact~\ref{bvnobin} is that in (2)
   there may be  a definable `truly' binary function \cite[Section
    4.2]{BaldwinVer}  but it cannot be commutative.

We now show the versatility of the method of construction by finding
Steiner systems which both do and don't admit first order definable unary
functions.   Of course, since we are dealing with idempotent quasigroups
they can't produce algebraic terms for unary functions. But both the
quasigroups and the Steiner systems might define unary functions with more
complicated definitions. We repeat a short argument from \cite{BaldwinPao}
to motivate the argument of Proposition~\ref{nounary}.

\begin{theorem}\label{realudc} If $\mu\in \Uscr$ and $\mu(\boldsymbol{\alpha}) \geq 2$,
 i.e. lines have length at least $4$, then there is a $12$-element linear
 $A$ that is $0$-primitive over a singleton $a\in A$. Moreover, if
 $\mu(A/\{a\}) =1$, then $T_\mu$ has a non-identity definable unary
function.
\end{theorem}

\begin{proof}

Let $A$ be the $\tau$-structure in $\bK*_0$ with $12$ elements, $\{a,b,c\}
\cup \{d_i:1\leq i \leq 9\}$. Let $\boldsymbol{\eta}$ be the isomorphism
type of the pair $(\{a\}, \{b,c\} \cup \{d_i:1\leq i \leq
9\})=(\{a\},A-\{a\})$ where $R$ holds of $(a,b,c)$, $(a ,d_{2i+1},
d_{2i+2})$ (for $0\leq i \leq 3$), $(b, d_{2i+2}, d_{2i+3} )$ (for $0\leq i
\leq 2$), $(b, d_8,d_1)$, and finally each triple from $\{c, d_8, d_9,
d_4\}$.
 There are 12 points, nine $3$-point line segments and one with $4$ points
so  $\delta(A) = 12 -(9 + 2) = 1$. By inspection, each proper substructure
$A'$ has $\delta(A') \geq 2$ so $A$ is $1$-primitive over $\{a\}$.  But
$d_9$ is the unique point that is in exactly one clique within $A$. Thus,
if $\mu(\boldsymbol{\eta}) =1$, the formula $(\exists x_1, x_2, y_1,
\ldots, y_8) \Delta(x_0,x_1, x_2 ,y_1, \dots y_8,y_9)$ (where $\Delta$ is
the quantifier free diagram of $A$) defines $d_9$ over any $a$ in any model
of $T_\mu$ ($a$ determines $d_9$).  Since each element in the generic
$\Gscr_\mu$ is embeddable in a copy $A \subseteq \Gscr_\mu$, each model of
$T_\mu$ has a global definable unary function.
\end{proof}

 While we have given only one example, one can extend the length of the
 cycle and get infinitely many examples. Note that the construction in
 Theorem~\ref{realudc} is iterable so the definable closure \emph{ may not
be locally finite}.

 Recall $\Uscr$ is the set of $\mu$ such that for any good pair $B/A$,
$\mu(B/A) \geq \delta(B)$. While the construction with  $\Uscr$ as the
class $\bU$ of admissible $\mu$ does not imply trivial unary closure (for
any $X$, closure of $X$ under definable unary functions is $X$), we can
obtain triviality by taking the class $\bU$ of admissible $\mu$ as a
$\Cscr$ which we now define. Nothing changes from the construction in
Section~\ref{genspace} except now $\bU \rightarrow \Cscr$.

\begin{definition}\label{defCscr} Define $\Cscr$ by restricting $\Uscr $
by requiring for each good pair $(A/B)$ that if $|B| =1$ and some point in
$A$ is determined by $B$ (such as $d_9$ by $a$ in $\boldsymbol{\eta}$ in
Theorem~\ref{realudc}), then $\mu(A/B) =0$.
\end{definition}

 We used the  cycle graphs of \cite{CameronWebb} to prove in
\cite[4.11]{BaldwinPao} that there are $2^{\aleph_0}$ distinct strongly
minimal Steiner theories $T_\mu$; this proof remains valid if $\mu$ is
restricted to $\Cscr$ or even $\Cscr \cap \Tscr$
(Definition~\ref{tripdef}). The slight variant on the proof of amalgamation
to show $\Cscr \cap \Tscr$ is non-empty follows that in \cite[Lemma
5.1.2]{BaldwinsmssIII}.

Recall that $\bK_\mu$ is the instantiation for $\bL_\mu$ with $\bL_0$ taken
as $\bK_0$.

\begin{proposition}\label{nounary} If $\mu \in \Cscr \cap \Tscr$, and
$\Gscr_\mu$ is the generic
for $\bK_{\mu}$, for any $a\in M$, $\dcl(a) = \{a\}$.
\end{proposition}

\begin{proof} Clearly amalgamation can not introduce unary functions so we have a
generic $\Gscr_\mu$ with no unary functions. By completeness this holds for
all model of $T_\mu$.
\end{proof}

\begin{question}\label{uncoord} {\rm
We have found a multiplication $*$ with domain $M$ for each model $(M,R)$
of $T_\mu$. The particular $*$ depends on $M$, the choice of $V$ and free
choices made in the construction (i.e. on an enumeration $\eta$ of $M$).
\begin{enumerate}
 \item Are all the $(M, R,*)$ (for the same $V$) elementarily equivalent?
     in the same equationally complete variety? Each is  in a subvariety
     of $V$.
     \item  Do they represent continuum many distinct varieties? I.e, are
         the classes $HSP(\Gscr_\mu)$ distinct for (sufficiently)
         distinct $\mu$?
         \item What can be said about the  model theoretic complexity of
             (completions of) the various $\check T_{\mu,V}$?
           \end{enumerate}}
\end{question}

\section{Constructing strongly minimal quasigroups}\label{findqg}

We have shown that in  general the strongly minimal Steiner $k$-systems in
the vocabulary $\tau =\{R\}$ for $k >3$ do not define  quasigroups. More
precisely by Corollary~\ref{kq} and Lemma~\ref{pab} if a $k$-quasigroup in
definable in a model of $T_\mu$, then $k =q = p^n$ for some $n$.

Working in a vocabulary $\tau'=\{R, *\}$  we will  construct a generic
structure that is both a Steiner system $(M,R)$ and a quasigroup $(M,*)$.
We require that the $*$-algebra be in a given $(2,q)$-variety $V$
(Definition~\ref{rkdef}) that coordinatizes $(M,R)$. Then taking the reduct
of $(M,R,*)$ to the vocabulary containing only $*$ we have a strongly
minimal  quasigroup with a flat $\acl$-geometry.

Recall $\mu(\boldsymbol{\alpha}) = q-2$ implies the maximal cliques in
$\bK_\mu$ have at most $q$ elements; in the generic model they all have
$q$.  We next ensure this maximality condition holds on each finite
structure by restricting to a smaller class of $\tau$-structures, $\tilde
\bK_{q}$. There remain trivial lines ($2$ unrelated points which are thus
in no clique.).

To construct a strongly minimal quasigroup, we modify the setting of
Definition~\ref{BPsetting}, where the basic parameters  $\bL^*_0, \epsilon,
\bL_\mu,U$ became $\bK^*_0, \delta, \bK_\mu, \Uscr$ to construct  the
theory $T_\mu$.

\begin{definition}\label{defK'}
Fix $\mu \in \Uscr$ with $\mu(\boldsymbol{\alpha}) = q-2$ and a
$(2,q)$-variety $V$ of quasigroups.
\begin{enumerate}
\item Working with $\tau$-structures:
\begin{enumerate}

\item $\bL^*_0, \epsilon, \bL_0 \rightarrow \bK^*_0$, $\delta$, and
    $\bK_0$ are exactly as in Section~\ref{genspace}.
\item  A new intermediate step. Let $\tilde \bK_{q}   \subseteq
    \bK_\mu$ be the class of finite $\tau$-structures $(A,R)$ such that
    each maximal clique has $q$-elements. The restriction to `full
    lines'  is expressed by a single $\forall\exists$ $\tau$-sentence.
\end{enumerate}

    \item $\sigma \rightarrow \tau'$: Expand $\tau = \{R\}$ to $\tau'$ by
        adding a ternary relation symbol $H$.
\begin{enumerate}
     \item $\bL_0 \rightarrow \bK'_\mu$:    Let $ \bK'_\mu$ be the
         finite $\tau'$- structures $A'$ such that $A'\myrestriction
         \tau \in \tilde{\bK}_q$ and $A'\myrestriction H$ is the graph
         of $F_2(V)$ on each line.
        \item $\epsilon \rightarrow \delta'$: For any  $A'\in
            \bK'_\mu$, let $\delta'(A') = \delta(A'\myrestriction
            \tau)$. Define $\leq'$ from $\delta'$ as usual.  Note that
            each non-trivial line in $A'$ has $q$ elements.
\item $\bU \rightarrow \Uscr'$: See Definition~\ref{newgood}.(3).(b).

        \end{enumerate}
        \end{enumerate}
        \end{definition}

\medskip

\begin{construction}\label{constexp} {\rm For $q>3$ a prime power,
fix $\mu$ with  $\mu(\boldsymbol{\alpha}) = q-2$ and $V$ as
Definition~\ref{defK'}. For clarity, we label the $\tau'$-structures in
this argument with primes.

We construct from $C \in \bK_\mu$   a finite set of structures $C'_i \in
\bK'_\mu$ (for $i<r_C$). First, there is a canonical extension of $C$ to
$\tilde C \in \tilde{\bK}_q$.  Since there is a strong embedding of $C$
into the generic $\Gscr_\mu$, there is a structure $\tilde C \subseteq
\Gscr_\mu$ whose universe consists of extensions of each clique of length
at least $3$ in $C$ to have length $q$; but with no new intersections.  The
extensions exist because each line $\Gscr$ has length $q$. And there are no
intersections since $C \leq \Gscr$. Thus $\tilde C $ is a partial Steiner
system in $\bK_\mu$.

Now  we  construct a finite family of $\tau'$-expansions of $\tilde C$,
$\langle  C'_i: i< r_C\rangle$  by imposing on each non-trivial line $\ell
\subset \tilde C$ a copy of $F_2(V)$ with graph $H\restriction \ell$. Since
$C$ uniquely determines $\tilde C$, we denote these expansions by $C'_i$
(rather than $\tilde C'_i$.
There are finitely many non-isomorphic choices (depending on the
interaction of $H$ and $R$) to impose the $*$-structure on $\tilde C$;
$r_C$ is chosen to list all of them. By Lemma~\ref{2propagates} they all
are in $\bK'_0$ (satisfy the condition on $\delta$) and $C'_i
\myrestriction \tau = \tilde C$ for each $i$. }

\end{construction}

Since $\delta'$ ignores $*$,  $(A'/B')$ is a good pair  for $\bK'_\mu$  if
    the $\tau$-reduct of $(A'/B')$ is a good pair for $\bK_0$.

\begin{definition}[$\bK'$ good pairs and $\mu'$]\label{newgood}

\begin{enumerate}
\item Let $\boldsymbol{\alpha}'$ denote the    $\tau'$-isomorphism type
    $(a_1,a_2, b_1, \dots b_{k-2})$ of a $q$-element line over two
    points; $(\bbar/\abar)$ is good with respect to $\bK'_{0}$.
\item For a fixed $\bK_0$-good pair $(A/B)$ other than
    $\boldsymbol{\alpha}$, let $\langle \boldsymbol{\gamma'_i}\colon
    i<r_{A/B}\rangle$ for some $r_{A/B}\in \omega$, be a list of the
    isomorphism types of $\bK'_\mu$-good pairs $(A'/B')$ whose
    $\tau$-reduct is $(A/B)$ (with isomorphism type
    $\boldsymbol{\gamma}$).

\item For any  isomorphism type of a good pair $\boldsymbol{\gamma'}$  in
    $\tau'$
\begin{enumerate}
\item $\mu'(\boldsymbol{\alpha'}) =1$

\item  $\bU \rightarrow \Uscr'$: Let $\Uscr'$ be the collection of
    $\mu'$ such that for any other $\boldsymbol{\gamma'_i}$     (with
    $\boldsymbol{\gamma} = \boldsymbol{\gamma'_i}\myrestriction \tau$)
$$\mu'(\boldsymbol{\gamma_i}) =
    \mu(\boldsymbol{\gamma}).$$

    \item  $\bL_\mu \rightarrow \bK'_{\mu'}$:  Let $\bK'_{\mu'}$ be the
        class of structures $D'$ in $\bK'_{0}$ such that if $(A'/B')$
        is a good pair, then $\chi_{D'}(A'/B')  \leq \mu'(A'/B')$.
    \end{enumerate}
    \end{enumerate}
    \end{definition}

Note $\mu'(\boldsymbol{\alpha'})$ was forced to be $1$. As, when the
    $\tau$-reduct is a line of length $q$  over a two-point base, since two
    points determine a line and $V$ is a $(2,q)$ variety, all quasigroups
    on the line are isomorphic. Since, except for $\boldsymbol{\alpha'}$,
the various copies of each good pair have the same reduct to $\tau$ but may
differ in
 their quasigroup structure, each good pair $\boldsymbol{\gamma} = A/B$ in
$\bK_\mu$ has generated a finite number of distinct good pairs in
$\bK'_{\mu'}$.  With this framework in hand we can complete the proof of
Theorem~\ref{getsmquasigrp}. As the notation indicates, the variety $V$ has
not changed. But because $\tilde C$ does not expand uniquely to a
$\tau'$-structure $\mu'(\boldsymbol{\gamma'})$, in addition to having a
different domain will have a larger value than $\mu$ of the engendering
$\boldsymbol{\gamma}$. We show how to apply the proofs of the crucial
results 5.11 and 5.15 from \cite{BaldwinPao} for this result.

\begin{theorem}\label{getsmquasigrp}
For each $q$ and each $\mu \in \Uscr$, each of the $T_\mu$ in
Theorem~\ref{disp2} with line length $k = q =  p^n$ (for some $n$), and any
Mikado $(2,q)$ variety (block algebras) of quasigroups $V$ (i.e. block
algebras), there is a strongly minimal theory of quasigroups $T_{\mu',V}$
that defines a class of strongly minimal Steiner $q$-systems. The
associated quasigroups are not commutative.
\end{theorem}

\begin{proof} Choose $\mu\in \Uscr'$ as in Definition~\ref{newgood}.
We can  construct a generic, provided we prove amalgamation for
$\bK'_{\mu'}$. We now show that the amalgamation for the $\tau$-class, as
in Lemma 5.11
and Lemma 5.15
of \cite{BaldwinPao} yields an amalgamation for $\tau'$. Consider a triple
$D', E', F'$ in $\bK'_{\mu'}$ with $D'\subseteq F'$ and with $E'$
$0$-primitive over $D'$.  We put primes on the labels in \cite{BaldwinPao}
and omit the primes for the reducts to $\tau$.

Apply Lemma 5.11 and Lemma 5.15 of \cite{BaldwinPao} to obtain a
$\tau$-structure $G$ that solves the amalgamation in $\bK_\mu$. \emph{ A
priori} $G$ might not be the domain of a structure in $ \tilde{\bK}_q$.
However, since $E$ is primitive over $D$, although there may be a line
contained in the disjoint amalgam $G$ with two points in each of $D$ and
$F-D$, each line that contains 2 points in $E-D$ can contain at most one
from $D$.  Thus the expansion $\tilde G$ from Construction~\ref{constexp}
is $\tilde F \oplus_D \tilde E$. (The tilde is intentionally omitted from
$D$.) And the $\tau'$-amalgam $G'$ is obtained for $\tilde G$ as in
Construction~\ref{constexp}.  Since the definition of strong extension is
by omitting specified configurations, $\bK'_\mu$ has amalgamation and is
smooth. So there is a generic by Theorem~\ref{KL}.  Now the strong
minimality of the generic $\Gscr'_{\mu',V}$ follows exactly as in Lemmas
5.21 and 5.23
of \cite{BaldwinPao} and, letting $T_{\mu',V} =\th(\Gscr'_{\mu',V})$, we
have proved Theorem~\ref{getsmquasigrp}. By Fact~\ref{bvnobin}.(2), the
quasigroup cannot be commutative.
\end{proof}

Necessarily in the construction given, a  good pair  $(C/B)$ (other than
 $\boldsymbol{\alpha}$) of $\tau$-structures in the reduct of a model of
$T_{\mu',V}$ will have many (but finitely) more copies of $C$ over $B$ than
$\mu(C/B)$. Thus, $T_{\mu',V} \myrestriction \tau$ is not $T_\mu$. But it
is strongly minimal since there are fewer $\tau$-definable  than
$\tau'$-definable  sets and each is finite or cofinite.

\begin{remark}\label{reduct}
The reduct of $\Gscr'_{\mu',V}$ to $\{H\}$ is a strongly minimal block
algebra, in particular, a quasigroup. For ease of reading we give the
definition $R$  in the algebraic form of the reduct to $H$:

$$R(u,v,w)\leftrightarrow  \bigvee_{\sigma \in F_2(V)} w =\sigma(u,v).$$
The theory of that reduct is essentially $T_{\mu',V \myrestriction H}$.
\end{remark}

\begin{notation}\label{notV'} Let $V'\subseteq V$ be
the variety of quasigroups induced by $T_{\mu',V\myrestriction H}$. For
simplicity below, we will write $*$ for the multiplication with graph $H$
on models of $T'_{\mu',V}$. 
\end{notation}

A line of papers \cite{BaldwinShelahsatfree, KucPil, PilSkli} study the
model theoretic properties of free algebras that are saturated or in
particular categorical in their cardinality.  Although the  strong
minimality of the quasigroup constructed in Theorem~\ref{getsmquasigrp}
implies all these model theoretic conditions, we show it cannot be free. We
rely on \cite[Lemma 4.4.2]{BaldwinsmssIII}, which we rephrase (and slightly
correct from the 1st arxiv version) to use the notation here.

\begin{fact}\label{notbasis} If $M$ is a strongly minimal set constructed by the methods of
\cite{BaldwinPao}, $A \leq M$, and $|M-A|$ is infinite then there are
infinitely many elements of $M$ that are $*$ independent over $A$.
\end{fact}

 \begin{theorem}\label{notfree} The reduct $(M, *)$  of  a model $(M, *,R)$
 of $T_{\mu',V}$ is never $V'$-free on infinitely many generators.
   \end{theorem}

   \begin{proof} Suppose that $(M, *)$ is a free  $V'$-algebra with an infinite basis $Y$.
Then every permutation of $Y$ extends to an automorphism of $M$, so   $Y$
is an infinite   set of  indiscernibles in $M$. Let $X$ be an $\acl$-basis
of the strongly minimal set $(M,*)$. Then $X$ is also an infinite set of
indiscernibles. But it is immediate from strong minimality that there
     cannot be two distinct Ehrenfeucht-Mostowski types (for each $k$, the
     first order formulas that hold for each $k$ tuple from the set of
     indiscernibles) over the empty set realized by such sequences.
     However, the basis $Y$ cannot be algebraically independent in the
     model theoretic sense. If it were $Y$ would realize  the same
     $k$-types as $X$ for each $k$ and thus $Y\leq M$. But then by
     Fact~\ref{notbasis}, $M$ has infinite $V'$ dimension over $Y$, so $Y$
     is not a $V'$ basis.
\end{proof}

\begin{question}\label{BarCas}{\rm
The use of the graph of the quasigroup in Construction~\ref{constexp} is
similar to that in the study of model complete Steiner triple system of
Barbina and Casanovas \cite{BarbinaCasa}. As  noted in Remark 5.27 of
\cite{BaldwinPao}, their generic structure $M$ differs radically from ours:
for them, $\acl_M(X) = \dcl_M(X)$, the theory is at the other end of the
stability spectrum, and the generic model is atomic rather saturated. }

\emph{Is it possible to develop a theory of $q$-block algebras for
arbitrary prime powers similar to that for Steiner quasigroups with $q=3$
in their paper? That is, to find a model completion for each of the various
varieties of block algebras discussed in Definition~\ref{defrelvar}.3 ?}

{\rm \cite[Theorem 4.6]{GanWer} asserts that every $(m,q)$-variety has the
finite embedding property. Thus, the Fra{\" \i}ss\'{e} construction of a
model completion should be immediate.} \emph{Where do the resulting
theories lie in the stability classification?}

\end{question}

\begin{question}\label{incompleteness} {\rm
We guaranteed that the quasigroup $\Gscr_{\mu} \myrestriction \{*\}$ is in
$V$; but it may satisfy more equations. Different varieties of quasigroups
may have the same free algebra on two generators.
Construction~\ref{constexp} depends on both the original $\bK^q_{\mu}$ and
$F_2(V)$. \emph{ How many varieties can arise from the same $F_2(V)$? There
are two variants on this question. One is, `how many varieties of
quasigroup can have the same free algebra on two generators?'. The second
asks about only   varieties  that arise from choosing a $\mu$ and a variety
$V$ as in Construction~\ref{constexp}.}}
\end{question}

How do those varieties that the $\Gscr_{\mu',V}$ generate behave?
Immediately from known results each such variety satisfies the strong
properties listed below.

\begin{corollary}\label{uacor} For any model $M$ of any $T_{\mu',V}$
the reduct of $M$ to $*$ is in a  variety that is congruence permutable,
regular and uniform,  \cite[Theorem 3.1]{Quacksteiner} or \cite[Corollary
2.4]{GanWer}.
\end{corollary}

\begin{question}\label{uaques} Every finite algebra in a $(2,q)$-variety    has a finite
decomposition into directly irreducible algebras \cite[Corollary
2.4]{GanWer}. \emph{ Are there any similar results for infinite strongly
minimal block
 algebras?}
\end{question}

\subsection*{Acknowledgment} We acknowledge helpful discussions with Joel Berman, Omer Mermelstein,
Gianluca Paolini, and Viktor Verbovskiy. I thank the referee for a
particulary incisive and detailed report.

{\bf Compliance with ethical standards}

This paper is the work of the author.

Research partially supported by Simons travel grant  G3535.

Data sharing not applicable to this article as no datasets were generated
or analysed during the current study.

The author  has no competing interests to declare that are relevant to the
content of this article.

%
%

\end{document}